\documentclass[12pt]{elsarticle}
\usepackage{verbatim}
\usepackage{lineno,hyperref}
\usepackage{amssymb,amsmath,amsthm, amsfonts, mathrsfs, bbm, dsfont}
\usepackage[utf8]{inputenc}

\newtheorem{theorem}{Theorem}[section]
\newtheorem{lemma}{Lemma}[section]

\newtheorem{corollary}{Corollary}[section]

\newtheorem{definition}{Definition}[section]

\newtheorem{remark}{Remark}

\usepackage{extarrows}
\modulolinenumbers[5]
\usepackage[top=1in, bottom=1in, left=1.25in, right=1.25in]{geometry}

\usepackage{setspace}

\begin{document}
\begin{spacing}{1.1}
\begin{frontmatter}

\title{$N$-tuple wights noncommutative Orlicz spaces and some geometrical properties \tnoteref{mytitlenote}}
\tnotetext[mytitlenote]{The research has been supported by Research project of basic scientifific research business expenses of provincial colleges and universities in Hebei Province(2021QNJS11); Innovation and improvement project of academic team of Hebei University of Architecture Mathematics and Applied Mathematics (NO.
TD202006);  The Major Project of Education Department in Hebei (No. ZD2021039); Nature Science Foundation of Heibei Province under (No. A2019404009; China Postdoctoral Science Foundation (No. 2019M661047);Postdoctoral Foundation of Heibei Province under Grant B2019003016.}


\author[mymainaddress]{Ma Zhenhua  \corref{Ma Zhenhua}}
\cortext[Ma Zhenhua]{Corresponding author}
\ead{mazhenghua\_1981@163.com}

\author[mymainaddress]{Deng Quancai}
\cortext[Deng Quancai]{}
\ead{dqc\_com@163.com}

\address[mymainaddress]{Hebei University of Architecture, Zhangjiakou, 075024, P. R. China}

\begin{abstract}
This paper studies the $N$-tuple noncommutative Orlicz spaces $\bigoplus\limits_{j=1}^{n}L_{p,\lambda}^{(\Phi_{j})}(\widetilde{\mathcal{M}},\tau)$, where $L^{(\Phi_{j})}(\widetilde{\mathcal{M}},\tau)$ is noncommutative Orlicz spaces and $\widetilde{\mathcal{M}}$ is the $\tau$-measurable operators. Based on the maximum principle,  we give the Riesz-Thorin interpolation theorem on $\bigoplus\limits_{j=1}^{n}L_{p,\lambda}^{(\Phi_{j})}(\widetilde{\mathcal{M}},\tau)$ . As applications, the Clarkson inequality and some geometrical properties such as uniform convexity and unform smooth of noncommutative Orlicz space $L^{(\Phi_{s})}(\widetilde{\mathcal{M}},\tau), 0<s\leq1$ are given.
\end{abstract}


\end{frontmatter}

\linenumbers
\section{Preliminaries}
In 1936, for study uniform convexity of $L^{p}$ space, Clarkson gave some famous inequalities named Clarkson inequality \cite{Clarkson}.
In \cite{Xu},  the author used the noncommutative Riesz-Thorin interpolation theorem get the Clarkson inequality on noncommutative $L^{p}$ space. The principal objective of this paper is to investigate Riesz-Thorin interpolation theorem on noncommutative Orlicz spaces which yields the Clarkson inequality of noncommutative $L^{p}$ space. As applications, some geometrical properties such as uniform convexity and unform smooth of noncommutative Orlicz space $L^{(\Phi_{s})}(\widetilde{\mathcal{M}},\tau), 0<s\leq1$ are given.

The theory of Orlicz spaces associated to a trace was introduced by Muratov \cite{Muratov} and Kunze \cite{Kunze}.  Let $\mathcal{M}$ be a semi-finite von Neumann algebra acting on a Hilbert space $\mathcal{H}$ with a normal semi-finite faithful trace $\tau.$  A densely-defined closed linear operator $A: \mathcal{D}(A)\rightarrow \mathcal{H}$ with domain $\mathcal{D}(A)\subseteq\mathcal{H}$ is called affiliated with $\mathcal{M}$ if and only if $U^{\ast}AU=A$ for all unitary operators $U$ belonging to the commutant $\mathcal{M^{\prime}}$ of $\mathcal{M}$. Clearly, if $A\in \mathcal{M}$ then $A$ is affiliated with $\mathcal{M}$. If $A$ is a (densely-defined closed) operator affiliated with $\mathcal{M}$ and $A=U|A|$  the polar decomposition, where $|A|=(A^{\ast}A)^{\frac{1}{2}}$ and $U$ is a partial isometry, then $A$ said to be $\tau$-measurable if and only if there exists a number $\lambda\geq0$ such that $\tau(e_{(\lambda,\infty)}(|A|))<\infty$, where $e_{[0,\lambda]}$ is the spectral projection of $|A|$ and $\tau$ is the trace of normal faithful and semifinite. The collection of all $\tau$-measurable operators is denoted by $\widetilde{\mathcal{M}}$. The spectral decomposition implies that a von Neumann algebra $\mathcal{M}$ is generated by its projections.  Recall that an element $A\in \mathcal{M}_{+}$ is a linear combination of mutually orthogonal projections if $A=\sum\limits_{k=1}^{n}\alpha_{k}e_{k}$ with $\alpha_{k}\in \mathds{R}_{+}$ and projection $e_{k}\in \mathcal{M}$ such that $e_{k}e_{j}=0$ whenever $k\neq j$ \cite{Xu}.

Next we recall the definition and  some basic properties of noncommutative Orlicz spaces.

A function $\Phi: [0,\infty)\rightarrow[0,\infty]$ is called an Orlicz function
if and only if $\Phi(u)=\int^{|u|}_{0}p(t)dt$, where the right derivative $p$ of $\Phi$ satisfies $p$ is right-continuous and nondecreasing, $p(t)>0$ whenever $t>0$ and $p(0)=0$ with $\lim\limits_{t\rightarrow\infty}p(t)=\infty$ \cite{Chen}. Further we say an Orlicz function $\Phi$ satisfies the $\Delta_{2}$-condition for large $t$ (for small $t$, or for all $t$), written often as $\Phi\in\Delta_{2}$, if there exist constants $t_{0}>0, K>2$ such that $\Phi(2t)\leq K\Phi(t)$ for  $|t|\geq t_{0}$ \cite{Rao}. 

If $A\in\widetilde{\mathcal{M}}$ and $\Phi$ is an Orlicz function,  we denote $\widetilde{\rho}_{\Phi}(A)=\tau(\Phi(|A|))$, hence we can define a corresponding space, which is named the noncommutative Orlicz space, as follows:
$$L^{\Phi}(\widetilde{\mathcal{M}},\tau)=\{A\in\widetilde{\mathcal{M}}:  \tau(\Phi(\lambda|A|))<\infty \,\, for\,\,some\,\,\lambda>0\}.$$
Also we could define the subspace $$E^{\Phi}(\widetilde{\mathcal{M}},\tau)=\{A\in\widetilde{\mathcal{M}}:  \tau(\Phi(\lambda|A|))<\infty \,\, for\,\,any\,\,\lambda>0\}.$$
We equip these spaces with the Luxemburg norm
$$\|A\|_{(\Phi)}=\inf\{\lambda>0: \tau\left(\Phi\left(\frac{|A|}{\lambda}\right)\right)\leq1\}.$$
In the case of $\Phi(A)=|A|^{p},\,\,1\leq p<\infty$, $L^{(\Phi)}(\widetilde{\mathcal{M}},\tau)$ is nothing but the noncommutative space
$L^{p}(\widetilde{\mathcal{M}},\tau)
=\left\{A\in\widetilde{\mathcal{M}}: \tau\left(|A|^{p}\right)<\infty\right\}$ \cite{Zhenhua} and the Luxemburg norm generated by this function is expressed by the formula
$$\|A\|_{p}=\big(\tau(|A|^{p})\big)^{\frac{1}{p}}.$$

One can define another norm on  $L^{\Phi}(\widetilde{\mathcal{M}},\tau)$ as follows
$$\|A\|_{\Phi}=\sup\{\tau(|AB|): B\in  L^{\Psi}(\widetilde{\mathcal{M}},\tau)\,\, and\,\, \tau(\Psi(B))\leq1 \},$$
where $\Psi: [0,\infty)\rightarrow [0,\infty] $ is defined by $\Psi(u)=\sup\{uv-\Phi(v): v\geq 0\}$. Here we call $\Psi$ the complementary function of $\Phi$. In the following, we use $L^{(\Phi)}(\widetilde{\mathcal{M}},\tau)$ and $L^{\Phi}(\widetilde{\mathcal{M}},\tau)$ denote the Orlicz which equipped Luxemberg and Orlicz norm respectively. The same as $E^{(\Phi)}(\widetilde{\mathcal{M}},\tau)$ and $E^{\Phi}(\widetilde{\mathcal{M}},\tau)$.

For more information on the theory of noncommutative Orlicz spaces we refer the reader to \cite{Rashed,Rashed1,Ghadir,Kunze,Zhenhua,Muratov}.

\section{Riesz-Thorin interpolation theorem of Noncommutative Orlicz spaces}
In this section,  we will give the definition of $N$-tuple noncommutative Orlicz spaces, also give some norm inequalities. For research the Riesz-Thorin interpolation theorem, a equivalent definition of Luxmburg norm must be given. As a corollary, the Clarkson inequality of noncommutive $L^{p}$ space could be get. The main ideas and proof ideas in this article are derived from literatures \cite{Xu} and \cite{Rao}.


Now let $\mathcal{N}=\mathcal{M}\oplus\mathcal{M}\oplus\cdots \oplus \mathcal{M}$ be the $n$-th von Neumann algebra direct sum of $\mathcal{M}$ with it self. We know that $\mathcal{N}$ acts on the direct sum Hilbert space $\mathcal{H}\oplus\mathcal{H}\oplus\cdots \oplus \mathcal{H}$ coordinatewise: $$(A_{1},A_{2},\ldots, A_{n})(x_{1},x_{2},\ldots,x_{n})=\sum\limits^{n}_{j=1}A_{j}x_{n},$$ where $A_{j}\in \mathcal{M}, i=1,2,\ldots n.$
Then $\mathcal{N}_{+}=\mathcal{M}_{+}\oplus\mathcal{M}_{+}\oplus\cdots \oplus \mathcal{M}_{+}$. 

Define: $\upsilon: \mathcal{N}_{+}\rightarrow \mathds{C}$ by $\upsilon(A_{1},A_{2},\ldots, A_{n})=\sum\limits^{n}_{j=1}\lambda_{j}\tau(A_{j}),$ where $\lambda_{j}\geq0$ and  $\tau$ is a normal faithful normal faithful normalized trace on $\mathcal{M}$, then $\upsilon$ is a normal faithful normal faithful normalized trace on $\mathcal{N}$. Now we give the following definition:
\begin{definition}
Let $\Phi=(\Phi_{1},\Phi_{2},\ldots \Phi_{n})$ be an $n$-tuple of N functions $\Phi_{j}$. For each $p\geq1, \lambda_{j}\geq0,$ and $n$-tuple of wights $\lambda=(\lambda_{1},\ldots,\lambda_{n})$ consider the direct sum space, named $n$-tuple of wights noncommutative Orlicz spaces as follows: $$\bigoplus\limits_{j=1}^{n}E_{p,\lambda}^{(\Phi_{j})}
=\{A=(A_{1},A_{2},\ldots,A_{n}): A_{j}\in E^{(\Phi_{j})}(\widetilde{\mathcal{M}},\tau), 1\leq j\leq n\}$$
with norm $\|\cdot\|_{(\Phi),p,\lambda}$ defined for $A_{j}\in E^{(\Phi_{j})}(\widetilde{\mathcal{M}},\tau):$
\begin{equation}
\|A\|_{(\Phi),p,\lambda}=
\begin{cases} \left[\sum\limits_{j=1}^{n}\lambda_{j}\|A_{j}\|_{(\Phi_{j})}^{p}\right]^{\frac{1}{p}}, &1\leq p<\infty, \nonumber\\
\max\limits_{j}\|A_{j}\|_{(\Phi_{j})}, &p=\infty.
\end{cases}
\end{equation}
or the norm $\|\cdot\|_{\Phi,p,\lambda}$ defined in the same way as before in which $\|\cdot\|_{(\Phi_{j})}$ is replaced by the Orlicz norm $\|\cdot\|_{\Phi_{j}}$, denotes by $\bigoplus\limits_{j=1}^{n}E_{p,\lambda}^{\Phi_{j}}$. The same way, if $\Psi_{j}$ is the complementary N-function of $\Phi_{j}$, denotes by $\bigoplus\limits_{j=1}^{n}E_{q,\lambda}^{(\Psi_{j})}$ which equip with $\|\cdot\|_{(\Psi),q,\lambda}$ and $\bigoplus\limits_{j=1}^{n}E_{q,\lambda}^{\Psi_{j}}$ which equip with $\|\cdot\|_{\Psi,q,\lambda}$ for the same weights $\lambda=(\lambda_{1},\ldots,\lambda_{n})$ and $q=\frac{p}{p-1}$. The same way, we also could define $\bigoplus\limits_{j=1}^{n}L_{p,\lambda}^{(\Phi_{j})}$, $\bigoplus\limits_{j=1}^{n}L_{p,\lambda}^{\Phi_{j}}$, $\bigoplus\limits_{j=1}^{n}L_{q,\lambda}^{(\Psi_{j})}$,
$\bigoplus\limits_{j=1}^{n}L_{q,\lambda}^{\Psi_{j}}$, and the norm as before.
\end{definition}
\begin{remark}
By the Theorem 3.4 of \cite{Zhenhua}, if for any $1\leq j\leq n$ with $\Phi_{j}\in \Delta_{2}$, we have $\bigoplus\limits_{j=1}^{n}L_{p,\lambda}^{(\Phi_{j})}=\bigoplus\limits_{j=1}^{n}E_{p,\lambda}^{(\Phi_{j})}.$
\end{remark}
\begin{lemma}
If $A\in \bigoplus\limits_{j=1}^{n}L_{p,\lambda}^{(\Phi_{j})}$ and $B\in \bigoplus\limits_{j=1}^{n}L_{q,\lambda}^{\Psi_{j}}$, where $1\leq p<\infty$,  we have

$(1)$ If $\|A\|_{(\Phi),p,\lambda}\leq1$, then we have
$\upsilon(\Phi(A))\leq \|A\|_{(\Phi),p,\lambda}\cdot \delta_{1}$,  where $\delta_{1}=\left(\sum\limits^{n}_{j=1}\lambda_{j}\right)^{\frac{1}{q}}.$

$(2)$ If $\|A\|_{(\Phi),p,\lambda}>1$, then we have $\upsilon(\Phi(A))>\delta_{2}$, where $\delta_{2}=\left[\sum\limits^{n}_{j=1}\lambda^{p}_{j}\|A_{j}\|^{p}_{(\Phi_{j})}\right]^{\frac{1}{p}}.$

$(3)$ $\upsilon(AB)\leq\|A\|_{(\Phi),p,\lambda}\cdot\|B\|_{\Psi,q,\lambda}$.
\end{lemma}
\begin{proof}
$(1)$ If $\|A\|_{(\Phi),p,\lambda}\leq1$, then from Proposition 3.4 of \cite{Ghadir} and classical H$\rm\ddot{o}$lder inequality, we have
\begin{eqnarray*}
\upsilon(\Phi(A))&=&\sum\limits^{n}_{j=1}\lambda_{j}\tau(\Phi_{j}(A_{j}))\\
&=&\sum\limits^{n}_{j=1}\lambda_{j}^{\frac{1}{q}}\cdot\lambda_{j}^{\frac{1}{p}}\tau(\Phi_{j}(A_{j}))\\
&\leq&\left[\sum\limits^{n}_{j=1}\lambda_{j}\left[\tau(\Phi_{j}(A_{j}))\right)^{p}\right]^{\frac{1}{p}}\cdot \left(\sum\limits^{n}_{j=1}\lambda_{j}\right)^{\frac{1}{q}}\\
&\leq&\left[\sum\limits^{n}_{j=1}\lambda_{j}\|A_{j}\|_{(\Phi_{j})}^{p}\right]^{\frac{1}{p}}\cdot \delta_{1}\\
&=&\|A\|_{(\Phi),p,\lambda}\cdot\delta_{1}.
\end{eqnarray*}

$(2)$ If $\|A\|_{(\Phi),p,\lambda}>1$, then from Proposition 3.4 of \cite{Ghadir}, we have
\begin{eqnarray*}
\left[\upsilon(\Phi(A))\right]^{p}&=&\left[\sum\limits^{n}_{j=1}\lambda_{j}\tau(\Phi_{j}(A_{j}))\right]^{p}\\
&>&\left[\sum\limits^{n}_{j=1}\lambda_{j}\|A_{j}\|_{(\Phi_{j})}\right]^{p}\\
&\geq&\sum\limits^{n}_{j=1}\lambda^{p}_{j}\|A_{j}\|^{p}_{(\Phi_{j})},
\end{eqnarray*}
which means that $$\upsilon(\Phi(A))>\left[\sum\limits^{n}_{j=1}\lambda^{p}_{j}\|A_{j}\|^{p}_{(\Phi_{j})}\right]^{\frac{1}{p}}=\delta_{2}.$$

$(3)$ From Theorem 3.3 of \cite{Ghadir} and classical H$\rm\ddot{o}$lder inequality, we get that
\begin{eqnarray*}
\upsilon(AB)&=&\sum\limits^{n}_{j=1}\lambda_{j}\left|\tau(A_{j}B_{j})\right|\\
&\leq&\sum\limits^{n}_{j=1}\lambda_{j}^{\frac{1}{p}+\frac{1}{q}}\|A_{j}\|_{(\Phi)}\|B_{j}\|_{\Psi}\\
&\leq&\left(\sum\limits^{n}_{j=1}\lambda_{j}\|A_{j}\|^{p}_{(\Phi)}\right)^{\frac{1}{p}}
\cdot\left(\sum\limits^{n}_{j=1}\lambda_{j}\|B_{j}\|^{q}_{\Psi}\right)^{\frac{1}{q}}\\
&=&\|A\|_{(\Phi),p,\lambda}\cdot \|B\|_{\Psi,q,\lambda},
\end{eqnarray*}
\end{proof}
\begin{remark}
If $\Phi$ is 1-tuple N-function and $\lambda=1$, the Lemma 2.1 just be the Theorem 3.3 and Proposition 3.4 of \cite{Ghadir}.
\end{remark}
\begin{theorem}
If $A\in \bigoplus\limits_{j=1}^{n}E_{p,\lambda}^{(\Phi_{j})}$, then for $1\leq p<\infty$, the weighted norms $\|\cdot\|_{(\Phi),p,\lambda}$ is given by
$$\|A\|_{(\Phi),p,\lambda}=\sup\left\{\upsilon(AB):\|B\|_{\Psi,q,\lambda}\leq1\right\},$$
\end{theorem}
\begin{proof}
If $\|B\|_{\Psi,q,\lambda}\leq1$. One side, by (3) of the Lemma 2.1, we have
$$\upsilon(AB)\leq\|A\|_{(\Phi),p,\lambda}\cdot \|B\|_{\Psi,q,\lambda}\leq \|A\|_{(\Phi),p,\lambda}.$$
The other side, for simplicity, we may take that $\|A\|_{(\Phi),p,\lambda}=1$ and assume that $A_{j}\geq0$. Let $\{e_{jn}\}$ be the projection of $A_{j}$ and $0<\tau(e_{jn})<\infty$. By Proposition 3.4 of \cite{Ghadir}, for any $\varepsilon>0$, we have $\tau\left[\Phi_{j}\left((1+\varepsilon)A_{j}\right)\right]\geq\|(1+\varepsilon)A_{j}\|_{(\Phi_{j})}=1+\varepsilon$.

If we define the operator $A_{jm}=A_{j}(e_{j1}+e_{j2}+\cdots+e_{jm})(m\leq n)$, where $A_{j}=\sum\limits_{k=1}^{n}\alpha_{k}e_{jk}$ and $e_{jk}=0, k>n$, then $A_{jm}\uparrow A_{j}$ as $m\rightarrow\infty$, there exists an $m_{0}$ such that for $m\geq m_{0}$ one have $$\upsilon\left[\frac{1}{\delta_{2}}\Phi\left((1+\varepsilon)A_{m}\right)\right]=\sum\limits_{k=1}^{n}\frac{1}{\delta_{2}}\lambda_{j}\Phi_{j}((1+\varepsilon)A_{jm})\geq \left(1+\frac{\varepsilon}{2}\right).$$

If we set $$B_{jm}=\frac{\delta_{2}^{-1}p((1+\varepsilon)\lambda_{j}A_{jm})}{\delta_{1}\left(1+\tau(\Psi_{j}( \delta_{2}^{-1}p((1+\varepsilon)\lambda_{j}A_{jm})))\right)},$$
then $B_{jm}$ is bounded operators and $B_{m}\in \bigoplus\limits_{j=1}^{n}E_{q,\lambda}^{\Psi_{j}}$ for each $m$. Moreover by definition 1.7 of \cite{Zhenhua} and 1.9 of \cite{Chen} we have, $\|B_{jm}\|_{\Psi_{j},q,\lambda}\leq1$ since $\|A\|_{(\Phi),p,\lambda}=1$.

Hence, $$\|B_{m}\|_{\Psi,q,\lambda}=\left(\sum\limits^{n}_{j=1}\lambda_{j}\|B_{jm}\|^{q}_{\Psi_{j}}\right)^{\frac{1}{q}}\leq1.$$
However, one has
\begin{eqnarray*}
\sup\{\upsilon(AB)\}&=&\sup\left\{\sum\limits^{n}_{j=1}\lambda_{j}\tau(A_{j}B_{j}):B_{j}\in E^{\Psi_{j}},\|B\|_{\Psi,q,\lambda}\leq1\right\}\\
&\geq&\sup\limits_{m\geq m_{0}}\left\{\sum\limits^{n}_{j=1}\lambda_{j}\tau(A_{j} B_{jm}):B_{jm}\in E^{\Psi_{j}}, \|B_{m}\|_{\Psi,q,\lambda}\leq1\right\}\\
&\geq&\frac{1}{1+\varepsilon}\sup_{m\geq m_{0}}\left\{\sum\limits^{n}_{j=1}\tau((1+\varepsilon)\lambda_{j}A_{jm} B_{jm})\right\}\\
&=&\frac{1}{1+\varepsilon}\sup_{m\geq m_{0}}\left\{\sum\limits^{n}_{j=1}\frac{\tau(\Phi_{j}(1+\varepsilon)\lambda_{j}A_{jm})+\tau(\Psi_{j}(\delta_{2}^{-1}p(1+\varepsilon)\lambda_{j}A_{jm}))}{\delta_{1}\delta_{2}\left(1+\tau(\Psi_{j}(\delta_{2}^{-1}p(1+\varepsilon)\lambda_{j}A_{jm}))\right)}\right\}\\
&>&\frac{1}{1+\varepsilon},
\end{eqnarray*}
since $\varepsilon>0$ is arbitrary we get the desired inequality.
\end{proof}
\begin{definition}\cite{Cleaver}
Let $\Phi_{1}$ and $\Phi_{2}$ be N-functions and define $\Phi_{s}$ to be the inverse of $\Phi_{s}^{-1}(u)=[\Phi_{1}^{-1}(u)]^{1-s}[\Phi_{2}^{-1}(u)]^{s}$ for $0\leq s\leq1, u\geq0$, where $\Phi^{-1}$ is the unique inverse of the N-function $\Phi$.
\end{definition}
\begin{theorem}
Let $\Phi_{i}=(\Phi_{i1},\Phi_{i2},\ldots \Phi_{in}),Q_{i}=(Q_{i1},Q_{i2},\ldots Q_{in}),i=1,2$ be n-tuples of N-functions and $0\leq r_{1},r_{2},t_{1},t_{2}\leq\infty, \lambda=(\lambda_{1},\ldots,\lambda_{n})$  be given positive numbers. Next let $\Phi_{s}=(\Phi_{s1},\Phi_{s2},\ldots \Phi_{sn}),Q_{s}=(Q_{s1},Q_{s2},\ldots Q_{sn})$ be the associated intermediate N-functions, $$\frac{1}{r_{s}}=\frac{1-s}{r_{1}}+\frac{s}{r_{2}}, \frac{1}{t_{s}}=\frac{1-s}{t_{1}}+\frac{s}{t_{2}}, 0\leq s\leq1.$$
If $T:\bigoplus\limits_{j=1}^{n}E^{(\Phi_{ij})}_{r_{i},\lambda}\rightarrow \bigoplus\limits_{j=1}^{n}L^{(Q_{ij})}_{t_{i},\lambda}$ is a bounded linear operator with bounds $K_{1},K_{2}$, such that
$\|TA\|_{(Q_{i}),t_{i},\lambda}\leq K_{i}\|A\|_{(\Phi_{i}),r_{i},\lambda}, A\in \bigoplus\limits_{j=1}^{n}E^{(\Phi_{ij})}_{r_{i},\lambda}, i=1,2$.

Then $T$ is also defined on $\bigoplus\limits_{j=1}^{n}E^{(\Phi_{sj})}_{r_{s},\lambda}$ into $\bigoplus\limits_{j=1}^{n}L^{(\Phi_{sj})}_{r_{s},\lambda}$ for all $0\leq s\leq1$ and one have the bound
$$\|TA\|_{(Q_{s}),t_{s},\lambda}\leq K_{1}^{1-s}K_{2}^{s}\|A\|_{(\Phi_{s}),r_{s},\lambda},$$ where $A\in \bigoplus\limits_{j=1}^{n}E^{(\Phi_{sj})}_{r_{s},\lambda}$.
\end{theorem}
\begin{proof}
Let $A=(A_{1},A_{2},\ldots,A_{n})\in \bigoplus\limits_{j=1}^{n}E^{(\Phi_{sj})}_{r_{s},\lambda}, B=(B_{1},B_{2},\ldots,B_{n})\in\bigoplus\limits_{j=1}^{n}E^{\Psi_{tj}}_{t_{s},\lambda}$ with polar decompositions  $A_{k}=U_{k}|A_{k}|,B_{k}=V_{k}|B_{k}|$. Assume that $\|A\|_{(\Phi_{i}),r_{s},\lambda}\leq1, \|B\|_{\Psi_{i},t_{s},\lambda}\leq1$ where $|A_{k}|=\sum\limits_{j=1}^{n}\alpha_{j}e_{kj}, |B_{k}|=\sum\limits_{j=1}^{n}\beta_{j}e'_{kj}$.

Define for $z=\mathds{C}$ and $k=1,2,\ldots,n$
$$A(z)=(A_{1}(z),A_{2}(z),\ldots,A_{n}(z))$$  and $$B(z)=(B_{1}(z),B_{2}(z),\ldots,B_{n}(z)),$$
where
$$A_{k}(z)=U_{k}\Phi_{sk}\left[(\Phi_{1k}^{-1})^{1-z}(\Phi_{2k}^{-1})^{z}\right](|A_{k}|),$$
$$B_{k}(z)=V_{k}\Psi_{sk}\left[(\Psi_{1k}^{-1})^{1-z}(\Psi_{2k}^{-1})^{z}\right](|B_{k}|).$$
Then,
\begin{eqnarray*}
A_{k}(z)&=&U_{k}\Phi_{sk}\left[\left(\Phi_{1k}^{-1}\left(\sum_{j=1}^{n}\alpha_{j}e_{kj}\right)\right)^{1-z}\left(\Phi_{2k}^{-1}\left(\sum\limits_{j=1}^{n}\alpha_{j}e_{kj}\right)\right)^{z}\right]\\
&=&\sum\limits_{j=1}^{n}\Phi_{sk}\left[\left(\Phi_{1k}^{-1}(\alpha_{j})\right)^{1-z}\left(\Phi_{2k}^{-1}(\alpha_{j})\right)^{z}\right]U_{k} e_{kj}
\end{eqnarray*}
Hence, $z\rightarrow A(z)$ is an analytic function on $\mathds{C}$ with value in $\widetilde{\mathcal{M}}$. The same reduction applies to $B$.

Now we could define a bounded entire function $$H(z)=K_{1}^{z-1}K_{2}^{-z}\tau(B(z)TA(z)).$$
If $z=it$ for $t\in \mathds{R}$, we have
\begin{eqnarray*}
A_{k}(it)&=&\sum\limits_{j=1}^{n}\Phi_{sk}\left[\Phi^{-1}_{sk}(\alpha_{j})\right]U_{k} e_{kj}\\
&=&\sum\limits_{j=1}^{n}\Phi_{sk}[\left(\Phi_{1k}^{-1}(\alpha_{j})\right)^{1-it}(\Phi_{2k}^{-1}(\alpha_{j}))^{it}]U_{k} e_{kj}\\
&=&\sum\limits_{j=1}^{n}\Phi_{sk}\left[\left(\frac{\Phi_{2k}^{-1}(\alpha_{j})}{\Phi_{1k}^{-1}(\alpha_{j})}\right)^{it}\right]U_{k} e_{kj}\cdot \sum\limits_{j=1}^{n}\Phi_{sk}\left[\Phi_{1k}^{-1}(\alpha_{j})\right]U_{k} e_{kj}\\
&=&\left[\Phi_{sk}\left(\frac{\Phi_{2k}^{-1}}{\Phi_{1k}^{-1}}(|A_{k}|)\right)\right]^{it}\cdot \Phi_{sk}\left(\Phi_{1k}^{-1}(|A_{k}|)\right).
\end{eqnarray*}
Hence,$$|A_{k}(it)|^{2}=A_{k}(it)^{*}A_{k}(it)=\left[\Phi_{sk}\left(\Phi_{1k}^{-1}(|A_{k}|)\right)\right]^{2}$$ which means $$|A_{k}(it)|=\Phi_{sk}\left(\Phi_{1k}^{-1}(|A_{k}|)\right).$$
Hence for any $1\leq k\leq n$ we have $\tau(\Phi_{1k}(A_{k}(it)))=\tau(\Phi_{sk}(A_{k}))$ which implies that $$\|A_{j}(it)\|_{(\Phi_{1j})}=\|A_{j}\|_{(\Phi_{sj})}$$ and
\begin{eqnarray*} \upsilon(\Phi_{1}(A(it)))&=&\sum\limits_{j=1}^{n}\lambda_{j}\tau\left[\Phi_{1j}\left[\Phi_{sk}\left(\Phi_{1j}^{-1}(|A_{j}|)\right)\right]\right]\\
&=&\lambda_{1}\tau(\Phi_{s1}(|A_{1}|))+\lambda_{2}\tau(\Phi_{s2}(|A_{2}|))+\ldots+\lambda_{n}\tau(\Phi_{sn}(|A_{n}|))\\
&=&\upsilon(\Phi_{s}(|A|)),
\end{eqnarray*}
we get that $$\|A(it)\|_{(\Phi_{1}),r_{s},\lambda}=\|A\|_{(\Phi_{s}),r_{s},\lambda}\leq1.$$
Similar $\|B(it)\|_{\Psi_{1},t_{s},\lambda}=\|B\|_{\Psi_{s},t_{s},\lambda}\leq1$.
Thus by (3) of the Lemma 2.1 and the assumption on $T$, we have
$$|\tau(B(it)TA(it))|\leq  K_{1}\|B(it)\|_{\Psi_{1},t_{s},\lambda}\|A(it)\|_{(\Phi_{1}),r_{s},\lambda}\leq K_{1}.$$
It then follows that $|H(it)|\leq 1$ for any $t\in \mathds{R}$. In the same way, we show $|H(1+it)|\leq 1$. Therefore, by the maximum principle, for any $\theta\in \mathds{C}$, we get $$|H(\theta)|=|K_{1}^{\theta-1}K_{2}^{-\theta}\tau(B(\theta)TA(\theta)|\leq1.$$
Hence, $$|\tau(BTA)|\leq K_{1}^{1-\theta}K_{2}^{\theta}$$
By the Theorem 2.1 we could get that $$\|TA\|_{(Q_{s}),r_{s},\lambda}\leq K_{1}^{1-\theta}K_{2}^{\theta}\|A\|_{(\Phi_{s}),r_{s},\lambda}.$$
\end{proof}

\begin{theorem}
Let $\Phi$ be an N-function and $\Phi_{s}$ be the inverse which satisfies that $\Phi_{s}^{-1}(u)=\left[\Phi^{-1}(u)\right]^{1-s}\left[\Phi_{0}^{-1}(u)\right]^{s}
=\left[\Phi^{-1}(u)\right]^{1-s}u^{\frac{s}{2}}$ where $0< s\leq1$ and $\Phi_{0}(u)=u^{2}$. If $L^{(\Phi)}(\widetilde{\mathcal{M}},\tau)$ is the noncommutative Orlicz space, then we have for $A,B\in L^{(\Phi_{s})}(\widetilde{\mathcal{M}},\tau)$:
$$\left(\|A+B\|_{(\Phi_{s})}^{\frac{2}{s}}+\|A-B\|_{(\Phi_{s})}^{\frac{2}{s}}\right)^{\frac{s}{2}}
\leq 2^{\frac{s}{2}}\left(\|A\|_{(\Phi_{s})}^{\frac{2}{2-s}}+\|B\|_{(\Phi_{s})}^{\frac{2}{2-s}}\right)^{\frac{2-s}{2}}.$$
\end{theorem}
\begin{proof}
Let $\Phi_{1}=(\Phi,\Phi)$ be the 2-vector of N-functions, $\lambda=(1,1),1\leq r_{1}\leq\infty$ and set
$$\bigoplus\limits_{j=1}^{2}E_{r_{1}}^{(\Phi)}(\widetilde{\mathcal{M}},\tau)
=\{(A,B): A,B\in E^{(\Phi)}(\widetilde{\mathcal{M}},\tau), \|(A,B)\|_{(\Phi_{1}),r_{1}}<\infty\},$$
where
\begin{equation}
\|(A,B)\|_{(\Phi_{1}),r_{1}}=
\begin{cases} \left[\|A\|_{(\Phi)}^{r_{1}}+\|B\|_{(\Phi)}^{r_{1}}\right]^{\frac{1}{r_{1}}}, &1\leq r_{1}<\infty, \nonumber\\
\max\{\|A\|_{(\Phi)},\|B\|_{(\Phi)}\}, &r_{1}=\infty.
\end{cases}
\end{equation}
Take $Q_{1}=\Phi_{1}=(\Phi,\Phi)$ and $Q_{2}=\Phi_{2}=(\Phi_{0},\Phi_{0})$ where $\Phi_{0}(u)=u^{2}$.

Set $r_{1}=1,r_{2}=t_{2}=2$ and $t_{1}=+\infty$. Define the linear operator $T: \bigoplus\limits_{j=1}^{2}E^{(\Phi_{i})}_{r_{i}}\rightarrow \bigoplus\limits_{j=1}^{2}L^{(Q_{i})}_{t_{i}}$ by the equation $T(A,B)=(A+B,A-B)$, we then have
\begin{eqnarray*} \|T(A,B)\|_{(Q_{1}),t_{1}}&=&\max\{\|A+B\|_{(\Phi)},\|A-B\|_{(\Phi)}\}\\
&\leq& \|A\|_{(\Phi)}+\|B\|_{(\Phi)}\\
&=&K_{1}\|(A,B)\|_{(\Phi_{1}),r_{1}}.
\end{eqnarray*}
Hence,$K_{1}=1$  and since $\|\cdot\|_{(\Phi_{0})}=\|\cdot\|_{2}$, we find
\begin{eqnarray*} \|T(A,B)\|_{(Q_{2}),t_{2}}&=&\left[\|A+B\|^{2}_{2}+\|A-B\|^{2}_{2}\right]^{\frac{1}{2}}\\
&=& \sqrt{2}\left[\|A\|^{2}_{2}+\|B\|^{2}_{2}\right]^{\frac{1}{2}}\\
&=&K_{2}\|(A,B)\|_{(\Phi_{2}),r_{2}}.
\end{eqnarray*}
Thus $K_{2}=\sqrt{2}$. Let $r_{s}$ and $t_{s}$ be given by $$\frac{1}{r_{s}}=\frac{1-s}{r_{1}}+\frac{s}{r_{2}}, \frac{1}{t_{s}}=\frac{1-s}{t_{1}}+\frac{s}{t_{2}}$$
then we have, $r_{s}=\frac{2}{2-s}, t_{s}=\frac{2}{s}$.

By the results of Theorem 2.2,
$$\|T(A,B)\|_{(Q_{s}),t_{s}}\leq 2^{\frac{s}{s}}\|(A,B)\|_{(\Phi_{s}),r_{s}}$$
since $K_{1}^{1-s}K_{2}^{s}=2^{\frac{s}{2}}$. Hence, we have
$$\|(A,B)\|_{(Q_{s}),r_{s}}=\left[\|A\|_{(\Phi_{s})}^{\frac{2}{2-s}}+\|B\|_{(\Phi_{s})}^{\frac{2}{2-s}}\right]^{\frac{2-s}{2}}$$
and
$$\|T(A,B)\|_{(Q_{s}),t_{s}}=\left(\|A+B\|_{(\Phi_{s})}^{\frac{2}{s}}+\|A-B\|_{(\Phi_{s})}^{\frac{2}{s}}\right)^{\frac{s}{2}}$$ which we could get the result.
\end{proof}
The following corollary is Clarkson inequality of noncommutative $L^{p}$ space and proof the process is completely similar to the P42 of \cite{Rao}.
\begin{corollary}
Suppose that $1<p<\infty$ and $q=\frac{p}{p-1}$. Then for $A,B\in L^{p}(\widetilde{\mathcal{M}},\tau)$, we have
$$\left(\|A+B\|_{p}^{q}+\|A-B\|_{p}^{q}\right)^{\frac{1}{q}}\leq2^{\frac{1}{q}}\left(\|A\|_{p}^{p}+\|B\|_{p}^{p}\right)^{\frac{1}{p}}, 1<p\leq2,$$
and
$$\left(\|A+B\|_{p}^{p}+\|A-B\|_{p}^{p}\right)^{\frac{1}{p}}\leq2^{\frac{1}{p}}\left(\|A\|_{p}^{q}+\|B\|_{p}^{q}\right)^{\frac{1}{q}}, 2\leq p\leq \infty.$$
\end{corollary}
\begin{proof}
If $1<p\leq2$, let $1<\alpha<p\leq2$ and $\Phi(u)=|u|^{\alpha}, \Phi_{0}(u)=|u|^{2}, s=\frac{2(p-\alpha)}{p(2-\alpha)}$. Then $0<s\leq1$ and $\Phi_{s}^{-1}(u)=|u|^{\frac{1}{p}}$ or $\Phi_{s}(u)=|u|^{p}$. Hence $\|\cdot\|_{(\Phi_{s})}=\|\cdot\|_{(p)}$ and since $\lim\limits_{\alpha\downarrow1}\frac{2}{s}=\frac{p}{p-1}=q$; $\lim\limits_{\alpha\downarrow1}\frac{2-s}{2}=\frac{1}{p}$ by the Theorem 2.3  we get the first inequality.

Similar let $2\leq p<\beta<\infty$  and $\Phi(u)=|u|^{\beta}, \Phi_{0}(u)=|u|^{2}, s=\frac{2(\beta-p)}{p(\beta-2)}$. Then $0\leq s\leq1$ and $\Phi_{s}(u)=|u|^{p},$ $\lim\limits_{\beta\uparrow\infty}\frac{2}{s}=p$; $\lim\limits_{\beta\uparrow\infty}\frac{2-s}{2}=\frac{1}{q}$, by the Theorem 2.3  we get the second inequality.
\end{proof}
\section{Some geometrical properties}
This section we contains some geometrical properties of noncommutative Orlicz spaces. These include uniform convexity, uniform smoothness which generalize the results of  noncommutative $L^{p}$ spaces. All these properties are based on Clarkson inequalities.
\begin{definition}\cite{Lindenstrauss}
Let $X$ be a Banach space. We define its modulus of convexity by
$$\delta_{X}(\varepsilon)=\inf\left\{1-\left\|\frac{x+y}{2}\right\|: x,y\in X, \|x\|=\|y\|=1, \|x-y\|=\varepsilon\right\}, 0<\varepsilon<2$$
and its modulus of smoothness by
 $$\rho_{X}(t)=\sup\left\{\frac{\|x+ty\|+\|x-ty\|}{2}-1: x,y\in X, \|x\|=\|y\|=1\right\}, t>0.$$
 $X$ is said to be uniformly convex if $\delta_{X}(\varepsilon)>0$ for every $2\geq\varepsilon>0$, and uniformly smooth if $\lim\limits_{t\rightarrow0}\frac{\rho_{X}(t)}{t}=0.$
\end{definition}
\begin{theorem}
Let $\Phi$ be an N-function and $\Phi_{s}$ be the inverse which satisfies that $\Phi_{s}^{-1}(u)=\left[\Phi^{-1}(u)\right]^{1-s}\left[\Phi_{0}^{-1}(u)\right]^{s}
=\left[\Phi^{-1}(u)\right]^{1-s}u^{\frac{s}{2}}$ where $0< s\leq1$ and $\Phi_{0}(u)=u^{2}$, then we have for $0<\varepsilon\leq2,$
$$\delta_{L^{(\Phi_{s})}}(\varepsilon)\geq 1-\frac{1}{2}\left[2^{\frac{2}{s}}-\varepsilon^{\frac{2}{s}}\right]^{\frac{s}{2}}$$
and
$$\rho_{L^{(\Phi_{s})}(t)}\leq \left(1+t^{\frac{2}{2-s}}\right)^{\frac{2-s}{2}}-1.$$
\end{theorem}
\begin{proof}
First, if $\|A-B\|_{(\Phi_{s})}=\varepsilon$, then theorem  2.3 implies for $A,B\in L^{(\Phi_{s})}(\widetilde{\mathcal{M}},\tau)$,
$$\left(\|A+B\|_{(\Phi_{s})}^{\frac{2}{s}}+\varepsilon^{\frac{2}{s}}\right)^{\frac{s}{2}}
\leq 2^{\frac{s}{2}}\cdot2^{\frac{2-s}{2}}=2.$$
Hence,
$$1-\frac{1}{2}\|A+B\|_{(\Phi_{s})}\geq1-\frac{1}{2}\left[2^{\frac{2}{s}}-\varepsilon^{\frac{2}{s}}\right]^{\frac{s}{2}}.$$
Taking infimum of $\|A\|_{(\Phi_{s})}=\|B\|_{(\Phi_{s})}=1$ we can get the desired result and $L^{(\Phi_{s})}(\widetilde{\mathcal{M}},\tau)$ is uniform convexity if $0<\varepsilon\leq2$, and reflexive.

Second, if $\|A\|_{(\Phi_{s})}=\|B\|_{(\Phi_{s})}=1$, then since $\frac{2}{s}\geq2$,
\begin{eqnarray*} \left[\frac{1}{2}\left(\|A+tB\|_{(\Phi_{s})}+\|A-tB\|_{(\Phi_{s})}\right)\right]^{\frac{2}{s}}&\leq&\frac{1}{2}\left[\|A+tB\|^{\frac{2}{s}}_{(\Phi_{s})}+\|A-tB\|^{\frac{2}{s}}_{(\Phi_{s})}\right]\\
&\leq&\frac{1}{2}\left[2^{\frac{s}{2}}\left(\|A\|^{\frac{2}{2-s}}_{(\Phi_{s})}+\|tB\|^{\frac{2}{2-s}}_{(\Phi_{s})}\right)^{\frac{2-s}{2}}\right]^{\frac{2}{s}}\\
&=&\frac{1}{2}\left[2^{\frac{s}{2}}\left(1+t^{\frac{2}{2-s}}\right)^{\frac{2-s}{2}}\right]^{\frac{2}{s}}\\
&=&\left(1+t^{\frac{2}{2-s}}\right)^{\frac{2-s}{s}}.
\end{eqnarray*}
Hence,
$$\frac{1}{2}\left(\|A+tB\|_{(\Phi_{s})}+\|A-tB\|_{(\Phi_{s})}\right)-1\leq \left(1+t^{\frac{2}{2-s}}\right)^{\frac{2-s}{2}}-1.$$
Taking the supremum on the left we can get the conclusion. Since $t>0$, we have that $L^{(\Phi_{s})}(\widetilde{\mathcal{M}},\tau)$ is uniformly smooth.
\end{proof}
From corollary 2.1, we can easily get the following results which appeared on \cite{Xu}.
\begin{corollary}
Suppose that $1<p<\infty, q=\frac{p}{p-1}, 0<\varepsilon<\varepsilon$ and $t>0$. Then for $A,B\in L^{p}(\widetilde{\mathcal{M}},\tau)$, we have

$(1)$ If $1<p<2$, then
$$\delta_{L^{p}}(\varepsilon)\geq \frac{\varepsilon^{q}}{q\cdot2^{q}}\ \ and\ \ \rho_{L^{p}(t)}\leq \frac{t^{p}}{p}.$$

$(2)$ If $2<p<\infty$, then
$$\delta_{L^{p}}(\varepsilon)\geq \frac{\varepsilon^{p}}{p\cdot2^{p}}\ \ and\ \ \rho_{L^{p}(t)}\leq \frac{t^{q}}{q}.$$

$(3)$ $L^{p}(\widetilde{\mathcal{M}},\tau)$ is uniformly convex and uniformly smooth. Consequently its reflexive.
\end{corollary}
\section*{Acknowledgement} We want to express our gratitude to the referee for all his/her careful revision and suggestions which has improved the final version of this work.
\section*{References}

\bibliography{mybibfile}

\end{spacing}
\end{document}